\documentclass[a4paper,draft]{amsart}
\usepackage{amsmath, amssymb, url}
\usepackage{braket}
\usepackage{amsthm}

 \author[Y. Numata]{NUMATA, Yasuhide}
 \address[Numata]{Department of Mathematical Sciences\\
 Shinshu University\\
 3-1-1 Asahi, Matsumoto-shi, Nagano-ken, 390-8621, Japan.}
 \email{nu@math.shinshu-u.ac.jp}
 \thanks{This work was supported by JSPS KAKENHI Grant Number 25800009}
 \subjclass[2010]{18A99, 18A32, 18D99}
\keywords{Association schemes; Bose--Mesner algebras; Quasi-schemoids}
\title[Construction of schemoids]{Construction of schemoids from posets}

\usepackage{easywncy}
\newcommand{\eel}{\mathcyr{l}}

\newcommand{\CCC}{\mathcal{C}}

\theoremstyle{plain}   

\newtheorem{thm}{Theorem}[section]
\newtheorem{theorem}[thm]{Theorem}

\newtheorem{lemma}[thm]{Lemma}

\theoremstyle{remark}  

\newtheorem{remark}[thm]{Remark}

\newtheorem{example}[thm]{Example}
\newtheorem*{ackn}{Acknowledgments}

\theoremstyle{definition}  
\newtheorem{definition}[thm]{Definition}

\begin{document}
\begin{abstract}
A schemoid is 
a generalization of 
association schemes
from the point of view of small categories.
In this article,
we discuss schemoid structures for two kinds of small categories;
the canonical small category defined by a poset,
and another small category which arises a poset.
We also discuss the schemoid algebra,
that is an analogue of the Bose--Mesner algebra for an association scheme,
for them.
\end{abstract}
\maketitle

\newcommand{\Stab}{\operatorname{Stab}}
\newcommand{\Obj}{\operatorname{Obj}}
\newcommand{\Mor}{\operatorname{Mor}}
\newcommand{\Hom}{\operatorname{Hom}}
\newcommand{\id}{\operatorname{id}}
\newcommand{\FF}{\mathbb{F}}
\newcommand{\KK}{\mathbb{K}}
\newcommand{\defit}[1]{\textit{#1}}
\section{Introduction}
We call  a pair of a finite set $X$ and a partition
 $S=\Set{R_0,R_1,\ldots,R_n}$ of $X\times X$ 
an \defit{$n$-class association scheme} if
$S$ satisfies the following:
\begin{enumerate}
 \item  \label{item:as:unit-condition}
      $R_0=\Set{(x,x)|x\in X}$.
 \item  \label{item:as:duality-condition}
	If $R\in S$, then $\Set{(y,x)|(x,y)\in R}\in S$. 
 \item \label{item:as:schemoid-condition}
       For $i,j,k\in\Set{0,1,\ldots,n}$, there exists $p_{i,j}^{k}$
       such that 
 \begin{align*}
  p_{i,j}^{k}=\#\Set{((x,y),(y,z))\in R_i\times R_j} 
 \end{align*}
       for every $(x,z)\in R_k$.
\end{enumerate}
Association schemes 
are introduced by Bose and Shimamoto \cite{MR0048772}
in their study of design of experiments. 
Since we can regard association schemes
as a generalization of combinatorial designs, groups, and so on (see also \cite{BannaiIto}),
many authors study association schemes 
 from the view point of algebraic combinatorics.
 Bose and Mesner introduced
an algebra
which arises an association scheme \cite{MR0102157}.
The algebra is called the Bose--Mesner algebra,
and plays an important role for algebraic study for association schemes.
In \cite{MR3318282},
Kuribayashi and Matsuo 
introduced an association schemoid and a quasi-schemoid, 
which are generalizations of an association scheme
from the viewpoint of small categories.
A quasi-schemoid, we call it a schemoid for short in this paper, 
is defined to be a pair of small category 
$\CCC$ and a partition $S$ of the morphisms of $\CCC$ 
satisfying the following condition:
\begin{itemize}
 \item  $\Set{ (f,g)\in\sigma\times\tau | f\circ g=h}$
and
 $\Set{ (f,g)\in\sigma\times\tau | f\circ g=k}$ 
have the same cardinarity for $\sigma$, $\tau$, $\mu\in S$
and $h$, $k\in \mu$.
\end{itemize}
Let $(X,\Set{R_0,R_1,\ldots,R_n})$ be an $n$-class association scheme.
Consider the codiscrete goupoid $\CCC_X$ on $X$,
i.e., the small category such that 
$\Obj(\CCC_X)=X$ and $\Hom_{\CCC_X}(x,y)=\Set{(y,x)}$.
For $R_i\in S$,
we define $\sigma_i$ to be the set of morphisms $(x,y)$ in $\CCC_X$
such that $(x,y)\in R_i$.
Since the compoisitions of morphisms in $\CCC_X$ 
are defined by $(x,y)\circ (y,z)=(x,z)$,
the pair $(\CCC_X,\Set{\sigma_0,\sigma_1,\ldots,\sigma_n})$ 
is a schemoid.
Hence we can identify an association scheme with a schemoid.
Moreover we can also construct a schemoid from a coherent
configuration in the same manner.
A schemoid requires the condition which is analogue of 
Condition \ref{item:as:schemoid-condition} 
in the definition of
an association scheme.
An association schemoid is a schemoid
satisfying the other conditions in the definition of
an association scheme.
Kuribayashi and Matsuo  
also introduce a subalgebra of the category algebra
which arises a schemoid.
The algebra is an analogue of 
the Bose--Mesner algebra for an association scheme.
Kuribayashi \cite{MR3306876} and 
Kuribayashi--Momose \cite{arxiv:1507.01745}
develop homotopy theory for schemoids.

The purpose of this article
is to give examples of schemoids.
In this paper,
we discuss schemoid structures for two kinds of small categories;
the canonical small category obtained from a poset,
and another acyclic small category obtained from a poset.
The organization of this article is the following:
We define schemoids and schemoid algebras
in Section \ref{sec:def}.
In Section \ref{sec:poset},
we discuss schemoid structures on 
two kinds of small categories which arises a poset.

\begin{ackn}
 The author thanks anonymous referees for their helpful comments.
\end{ackn}

\section{Definition}
\label{sec:def}
Here we recall the definition of small categories, 
schemoids and schemoid algebras.

First we recall small categories and functors.
A \defit{small category} $\CCC$
is a quintuple of the set $\Obj(\CCC)$ of objects,
 the set $\Mor(\CCC)$ of morphisms,
 maps $s\colon \Mor(\CCC)\to\Obj(\CCC)$ and 
 $t\colon \Mor(\CCC)\to\Obj(\CCC)$,
 and the operation $\circ$ of composition,
satisfying the following properties:
For each morphism $f\in \Mor(\CCC)$,
$s(f)$ is called the \defit{source} of $f$,
and $t(f)$ is called the \defit{target} of $f$.
A morphism $f$ is called an \defit{endomorphism}
if $s(f)=t(f)$.
For $x,y\in \Obj(\CCC)$,
define $\Hom_{\CCC}(x,y)=\Set{f\in\Mor(\CCC)|s(f)=x,\ t(f)=y}$.
A sequence $(f_n,f_{n-1},\ldots,f_1)$ of morphisms
is called a \defit{nerve}
if $t(f_i)=s(f_{i+1})$ for $i=1,\ldots,n-1$.
The composition  $g\circ f$ is defined for
each nerve $(g,f)$ of length $2$.
The composition $g\circ f$ is in $\Hom_{\CCC}(x,z)$
for $g\in \Hom_{\CCC}(y,z)$ and $f\in \Hom_{\CCC}(x,y)$.
Moreover the operation satisfies $(h\circ g)\circ f=h\circ (g\circ f)$
for every nerve $(h,g,f)$ of length $3$.
For $x\in \Obj(\CCC)$, a morphism from $x$ to $x$ is called an
\defit{endomorphism} on $x$.
For each $x\in \Obj(\CCC)$,
there uniquely exists 
an endomorphism $\id_x$ on $x$ such that
$\id_x\circ f=f$ for every $f$ with $t(f)=x$
and 
$g\circ \id_x=g$ for every $g$ with $s(g)=x$.
The morphism $\id_x$ is called the \defit{identity} on $x$.

Let $\CCC$ and $\CCC'$ be small categories.
We call a pair $\varphi=(\varphi^{\Obj},\varphi^{\Mor})$
a \defit{functor} from $\CCC$ to $\CCC'$
if 
the map $\varphi^{\Obj}$  from $\Obj(\CCC)$ to $\Obj(\CCC')$
and 
the map $\varphi^{\Mor}$  from $\Mor(\CCC)$ to $\Mor(\CCC')$
satisfy the following:
\begin{enumerate}
\item $\varphi^{\Mor}(\id_x)=\id_{\varphi^{\Obj}(x)}$ for each $x\in\Obj(\CCC)$.
\item $\varphi^{\Mor}(f\circ g)=\varphi^{\Mor}(f)\circ \varphi^{\Mor}(g)$
      for all nerve $(f,g)$ of length $2$ in $\CCC$.
\end{enumerate}

Next we define schemoids.
In this paper, we define a schemoid as the pair of a small category
$\CCC$ and a map $\eel$  from the set $\Mor(\CCC)$ of morphisms to a set $I$.
For a map $\eel$, we obtain 
a partition $\Set{\eel^{-1}(\Set{i})|i\in I})$ of $\Mor(\CCC)$.
On the other hand,
for a partition $S$ of $\Mor(\CCC)$,
we obtain the canonical surjection from $\Mor(\CCC)$ to $S$.
Via this translation, 
the following definition is equivalent to the original definition of a schemoid.
\begin{definition}
\label{def:schemoid:c}
Let $\CCC$ be  a small category, 
$I$  a set,
and $\eel$ a map from the set $\Mor(\CCC)$ of morphisms in $\CCC$ to 
the set $I$.
For $i,j\in I$ and $h\in\Mor(\CCC)$,
we define $N_{h}^{i,j}$ to be
\begin{align*}
 \Set{ (f,g)\in \Mor(\CCC)\times \Mor(\CCC)| \begin{array}{c}\eel(f)=i,\\ \eel(g)=j,\\  f\circ g = h.\end{array}}  .
\end{align*}
We call the triple $(\CCC,I,\eel)$ a \defit{schemoid}
if 
\begin{align*}
\eel(h)=\eel(k) \implies
\text{$N_{h}^{i,j}$ and $N_{k}^{i,j}$ have the same cardinarity}
\end{align*}
for each $i,j\in I$ and $h,k\in\Mor(\CCC)$.
\end{definition}
For a morphism $f$ of a small category $\CCC$,
we write $\CCC_f$ to denote
the minimum subcategory of $\CCC$
such that $ \Mor(\CCC_f)$ contains
\begin{align*}
\Set{g\in \Mor(\CCC)| 
\text{$f_1\circ g\circ f_2=f$ for some $f_1,f_2\in \Mor(\CCC)$}}.
\end{align*}
Then we can show the following lemma.
\begin{lemma}
\label{lemma:bijectionimpliesschemoid}
Let $\CCC$ be a small category which does not contain
any endomorphism except identities.
Let $\eel$ be a map from the set $\Mor(\CCC)$ of morphisms to a set $I$.
If the following condition holds,
then $(\CCC, I,\eel)$ is a schemoid:
For all morphisms $f$ and $g$ such that $\eel(f)=\eel(g)$,
there exists a functor $\varphi_{f,g}$ from $\CCC_f$ to $\CCC_g$ such that
\begin{enumerate}
 \item $\varphi^{\Mor}_{f,g}$ is a bijection;
 \item  $\eel(f')=\eel(\varphi^{\Mor}_{f,g}(f'))$ for each morphism $f'$ in $\CCC_f$; and
 \item \label{cond:3} $\varphi^{\Mor}_{f,g}(f)=g$.
\end{enumerate}
\end{lemma}
\begin{proof}
Let $h,k\in \Mor(\CCC)$ satisfy $\eel(h)=\eel(k)$.
The map $\varphi^{\Mor}_{h,k}$ induces 
a bijection from $N_{h}^{i,j}$ to $N_{k}^{i,j}$
for each $i,j\in I$.
Hence the triple $(\CCC,I,\eel)$ is a schemoid.
\end{proof}

For a small category  $\CCC$ and a field $\KK$,
define $\KK[\CCC]$ to be the $\KK$-vector space
whose basis is $\Mor(\CCC)$.
We define the product by
\begin{align*}
g\cdot f =\begin{cases}
	   g\circ f &s(g)=t(f)\\
	   0&s(g)\neq t(f)
	  \end{cases}
\end{align*}
for $f,g\in\Mor(\CCC)$.
Moreover, 
for $\sum_{f\in\Mor(\CCC)}\alpha_f f$
and
$\sum_{g\in\Mor(\CCC)}\beta_g g\in\KK[\CCC]$,
we define the product of them by 
\begin{align*}
(\sum_{f\in\Mor(\CCC)}\alpha_f f)\cdot (\sum_{g\in\Mor(\CCC)}\beta_g g)=\sum_{f\in\Mor(\CCC)}\sum_{g\in\Mor(\CCC)}(\alpha_f\beta_g) f\cdot g.  
\end{align*}
If $\Obj(\CCC)$ is a finite set,
then $\KK[\CCC]$ is a $\KK$-algebra with the unit 
$\sum_{x\in\Obj(\CCC)}\id_x$.
Let $\eel$ be a map from $\Mor(\CCC)$ to a set $I$.
Assume that $\eel^{-1}(\Set{i})$ is finite for every $i\in I$.
For $i \in I$,
we define $\overline{i}$
to be $\sum_{f\colon \eel(f)=i} f\in \KK[\CCC]$.
We define $\KK(\CCC,I,\eel)$ to be 
the vector subspace of $\KK[\CCC]$
spanned by $\Set{\overline{i} | i \in I}$.
For a schemoid $(C,I,\eel)$ such that $\eel^{-1}(\Set{i})$ is finite for every $i\in I$,
 $\KK(\CCC,I,\eel)$ is a subalgebra of $\KK[\CCC]$.
(The subalgebra  $\KK(\CCC,I,\eel)$ may not have the unit.)
We call $\KK(\CCC,I,\eel)$ a \defit{schemoid algebra}.
\section{Schemoids constructed from posets}
\label{sec:poset}
Here we consider two kinds of small categories defined from a poset.
The prototypical example of them is a schemoid structure for the
$n$-th Boolean lattice $2^{[n]}$, i.e.,
the poset consisting of all subsets of $\Set{1,\ldots,n}$
ordered by inclusion.
For $X,Y\in 2^{[n]}$,
we can consider the set difference $X\setminus Y$.
In \ref{sec:poset:natural},
we discuss a poset with the operation which is analogue of the operation
of set difference.
The operation induces a schemoid structure for 
the canonical small category obtained from the poset.
On the other hand,
for $X,Y\in 2^{[n]}$ with $X\cap Y = \emptyset$,
a greater element  $X\cup Y$ than $X$ 
is obtained from $X$ by adding $Y$.
By an analogue of the operation,
we introduce an acyclic small category 
obtained from a poset with some conditions
in \ref{sec:poset:ranked}.
The category has also a schemoid structure.

\subsection{Posets as a small category}
\label{sec:poset:natural}
Let $P$ be a poset with respect to $\leq$.
For $x,y\in P$,
we define the interval $[x,y]$ from $x$ to $y$
by $[x,y]=\Set{z|x\leq z \leq y}$.
We can naturally regard the poset $P$ 
as the following small category $\CCC_{P}$:
the set $\Obj(\CCC_P)$ of objects is $P$ and 
the set $\Mor(\CCC_P)$ of morphisms is the relation $\geq$,
i.e., $\Set{(y,x)|x\leq y}\subset P\times P$.
For $x\leq y\in P$,
$\Hom_{\CCC_P}(x,y)$ consists of $(y,x)$.
For $(y,x)\in\Hom_{\CCC_P}(x,y)$ and $(z,y)\in\Hom_{\CCC_P}(y,z)$,
it follows by definition that $x\leq z$.
We define the composition $(z,y)\circ(y,x)$ by $(z,y)\circ(y,x)=(z,x)$.
For $x\in P$,
$\id_x$ is $(x,x)$.

Here we consider a poset $P$ 
with a difference operation $\delta$ defined as follows:
\begin{definition}
Let $o$ be an element in the poset $P$,
and 
$\delta$ a map from the set $\Set{(y,x)\in P^2|x\leq y}$ to $ P$.
We say that $\delta$ is a \defit{difference operation with the base point} $o$
if
there exists
a family 
\begin{align*}
\Set{\varphi_{x,y}\colon [x,y] \to [o,\delta(y,x)] | x\leq y} 
\end{align*}
 of maps
satisfying the following:
\begin{enumerate}
 \item 
       Each $\varphi_{x,y}$ is a bijection
       from the interval $[x,y]$ to the interval $[o,\delta(y,x)]$.
 \item 
       \label{item:difference:consistent}
       $\delta(o,\varphi_{x,y}(z))=\delta(x,z)$ for $x\leq z  \leq y$.
\end{enumerate}
\end{definition}
 Let $\delta$ be a difference operation of poset $P$.
 In this case, we have bijections $\varphi_{x,y}$.
 If we fix an interval $[x,y]$, then 
 we can translate each element  in the interval $[x,y]$
 into some interval from the base point $o$ via the bijection $\varphi_{x,y}$.
 Fix $x\in P$ and consider two intervals $[x,y]$ and $[x,y']$.
 For $z\in [x,y]\cap [x,y']$,
 it follows by Condition \ref{item:difference:consistent} that
 $\delta(o,\varphi_{x,y}(z))=\delta(o,\varphi_{x,y'}(z))$.
 In this sense,
 Condition \ref{item:difference:consistent} implies that
 the translation depends not on the interval but only on the minimum of the interval.

Since the difference operation induces functors $\varphi_{f,g}$
from $(\CCC_P)_f$ to $(\CCC_P)_g$,
Theorem \ref{thm:direfence} 
follows from Lemma \ref{lemma:bijectionimpliesschemoid}.
\begin{theorem}
\label{thm:direfence}
For a poset $P$ with the difference operation $\delta$,
the triple  $(\CCC_P,P,\delta)$ is a schemoid.
\end{theorem}

\begin{example}
Let $P$ be the $n$-th Boolean lattice,
i.e., $2^{[n]}$ ordered by inclusion.
For $x\leq y \in P$,
we define $\delta(y,x)$
to be $y\setminus x$.
The map $\delta$ is a difference operation with
the base point $\emptyset$.
Hence $(\CCC_P,P,\delta)$ is a schemoid.
In this case,
the schemoid algebra 
$\KK(\CCC_P,P,\delta)$
is isomorphic to
$\KK[x_i|i\in P]/(x_i^2|i\in P)$.
\end{example}



\begin{example}
\label{example:boolean}
Let $P$ be 
a Coxeter groups
ordered by the Bruhat order.
For $x\leq y \in P$,
we define $\delta(y,x)$
to be $yx^{-1}$.
The map $\delta$ is a difference operation with
the base point $\varepsilon$.
Hence $(\CCC_P,P,\delta)$ is a schemoid.
In this case,
the schemoid algebra $\KK(\CCC_P,P,\delta)$ 
is isomorphic to
 the NilCoxeter algebra.
\end{example}

\begin{example}
\label{ex:simcomp}
Let $\Delta$ be a simplicial complex on the vertex set $V$.
Consider the lattice $P$ of faces of the simplicial complex $\Delta$.
(We regard $\emptyset$ as a face of $\Delta$.)
For $x,y\in P$,
we define $\delta(y,x)$
to be $y\setminus x$.
The map $\delta$ is a difference operation with
the base point $o=\emptyset$.
Hence $(\CCC_P,P,\delta)$ is a schemoid.
Let $I_{\Delta}$ be an ideal of $\KK[x_i|i\in V]$
generated by $\Set{x_{v_1}\cdots x_{v_l}|\Set{v_1,\ldots,v_l}\not\in\Delta}$.
The quotient ring $\KK[x_i|i\in V]/I_{\Delta}$ is called the
 \defit{Stanley--Reisner ring}.
Let $\tilde{I}_\Delta=I_{\Delta}+(x_i^2|i\in V)$.
The schemoid algebra $\KK(\CCC_P,P,\delta)$
is isomorphic to
$\KK[x_i|i\in V]/\tilde I_{\Delta}$.
\end{example}

\begin{remark}
 In Appendix of \cite{arxiv:1507.01745},
 Kuribayashi and Momose discuss schemoids in Example \ref{ex:simcomp}
 from the point of view of the category theory.
\end{remark}

\subsection{Yet another small category obtained from posets}
\label{sec:poset:ranked}
Here we introduce another kind of small categories 
obtained from a poset.
We also introduce a schemoid structure for it.

Let $P$ be a poset with respect to $\leq$.
Assume that  the number of minimal elements in 
$\Set{z\in P| x\leq z,\ y\leq z}$ is $1$ or $0$
for each pair $x,y \in P$.
We write $x\vee y$ to denote the minimum element in
$\Set{z\in P| x\leq z,\ y\leq z}$
if $\Set{z\in P| x\leq z,\ y\leq z}\neq \emptyset$.
Assume the following conditions:
\begin{enumerate}
 \item  $P$ has the minimum element $\hat 0$.
 \item $P$ is a ranked poset with the rank function $\rho$.
 \item $\rho(x\vee y)\leq\rho(x)+\rho(y)$ for $x,y\in P$.
\end{enumerate}
We define a small category $\tilde P$ whose set of objects is $P$.
For $x,y,d\in P$ such that $\rho(y)=\rho(x)+\rho(d)$ and $y=x\vee d$,
we define a morphism $f^d_{x,y}$ from $x$ to $y$.
Or equivalently,
\begin{align*}
 \Hom_{\tilde P}(x,y)=
 \Set{f^d_{x,y}|
\begin{array}{c}
d\in P.\\ y=x\vee d.\\ \rho(y)=\rho(x)+\rho(d).
\end{array}
} 
\end{align*}
for $x,y\in P$.
If $f^c_{x,y}$ and $f^{d}_{y,z}\in \Mor(\tilde P)$,
then $\rho(y)=\rho(x)+\rho(c)$ and $\rho(z)=\rho(y)+\rho(d)$.
Hence $\rho(z)=\rho(x)+\rho(c)+\rho(d)$.
Since $z=x\vee(c\vee d)$,
we have $\rho(x)+\rho(c)+\rho(d)=\rho(z)\leq \rho(x)+\rho(c\vee d)$. 
On the other hand, $\rho(x)+\rho(c\vee d)\leq \rho(x)+\rho(c)+\rho(d)$
since $\rho(c\vee d)\leq \rho(c)+\rho(d)$.
Hence $\rho(z)=\rho(x)+\rho(c\vee d)=\rho(x)+\rho(c)+\rho(d)$.
Since $f^{d\vee c}_{x,z}$ is in $\Mor(\tilde P)$,
we define the composition $f^{d}_{y,z} \circ f^c_{x,y}$
to be $f^{d\vee c}_{x,z}$.
\begin{example}
\label{example:ext}
Let $P=\Set{0,1,2}\times \Set{0,1}$.
For $(x,y),(x',y')\in P$,
$(x,y)\leq (x',y')$ if and only if
$x\leq x'$ and $y\leq y'$.
In this case,
the set of morphisms of $\tilde P$ 
consists of the following:
\begin{gather*}
 f^{(0,1)}_{(0,0),(0,1)},
 f^{(0,1)}_{(1,0),(1,1)},
 f^{(0,1)}_{(2,0),(2,1)},\\
 f^{(1,0)}_{(0,0),(1,0)},
 f^{(1,0)}_{(0,1),(1,1)},\\
 f^{(2,0)}_{(0,0),(2,0)},
 f^{(2,0)}_{(0,1),(2,1)},\\
 f^{(1,1)}_{(0,0),(1,1)},\\
 f^{(2,1)}_{(0,0),(2,1)},
\end{gather*}
and identities.
See also Figure \ref{fig:morphism}.
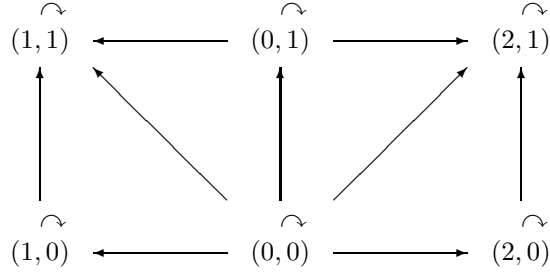
\begin{figure}
  \begin{center}
 \begin{picture}(200,110)
  \put(10,10){\makebox(0,0){$(1,0)$}}
  \put(10,90){\makebox(0,0){$(1,1)$}}
  \put(100,10){\makebox(0,0){$(0,0)$}}
  \put(100,90){\makebox(0,0){$(0,1)$}}
  \put(190,10){\makebox(0,0){$(2,0)$}}
  \put(190,90){\makebox(0,0){$(2,1)$}}
  \put(80,10){\vector(-1,0){50}}
  \put(120,10){\vector(1,0){50}}
  \put(100,30){\vector(0,1){50}}
  \put(80,30){\vector(-1,1){50}}
  \put(120,30){\vector(1,1){50}}
  \put(10,30){\vector(0,1){50}}
  \put(190,30){\vector(0,1){50}}
  \put(80,90){\vector(-1,0){50}}
  \put(120,90){\vector(1,0){50}}
  \put(10,20){\makebox(0,0)[bl]{$\curvearrowright$}}
  \put(10,100){\makebox(0,0)[bl]{$\curvearrowright$}}
  \put(100,20){\makebox(0,0)[bl]{$\curvearrowright$}}
  \put(100,100){\makebox(0,0)[bl]{$\curvearrowright$}}
  \put(190,20){\makebox(0,0)[bl]{$\curvearrowright$}}
  \put(190,100){\makebox(0,0)[bl]{$\curvearrowright$}}
 \end{picture}
    \caption{The small category in Example \ref{example:ext}}
    \label{fig:morphism}
  \end{center}
\end{figure}
\end{example}

\begin{theorem}
\label{thm:extendedposet}
For the map $\eel$ from $\Mor(\tilde P)$ to $P$
defined by $\eel(f^d_{x,y})=d$, 
the triple $(\tilde P,P,\eel)$ is a schemoid. 
\end{theorem}
\begin{proof}
 Let $f_{x,x\vee d}^d$ and $f_{x',x'\vee d}^d$ be in $\Mor(\tilde P)$.
 In this case, it follows from the definition of morphisms in $\tilde P$ that
 \begin{align*}
  \rho(x\vee d)&=\rho(x)+\rho(d),\\
  \rho(x'\vee d)&=\rho(x')+\rho(d).
 \end{align*}
 Let $f_{x,x\vee d}^d=f_{x\vee d_1,x\vee d_1\vee d_2}^{d_2}\circ f_{x,x\vee d_1}^{d_1}$.
 In this case, we have $d=d_1\vee d_2$ and
 $\rho(d)=\rho(d_1)+\rho(d_2)$.
 If $\rho(x'\vee d_1)<\rho(x')+\rho(d_1)$ or 
 $\rho(x'\vee d_1\vee d_2)<\rho(x'\vee d_1)+\rho(d_2)$,
 then we have
 \begin{align*}
  \rho(x'\vee d_1\vee d_2)&<\rho(x')+\rho(d_1)+\rho(d_2)
  =\rho(x')+\rho(d_1\vee d_2),
 \end{align*}
which contradicts $  \rho(x'\vee d)=\rho(x')+\rho(d)$.
Hence  morphisms $f_{x'\vee d_1,x'\vee d_1\vee d_2}^{d_2}$
and $f_{x',x'\vee d_1}^{d_1}$ 
satisfies
$f_{x',x'\vee d}^d= f_{x'\vee d_1,x'\vee d_1\vee d_2}^{d_2}\circ f_{x',x'\vee d_1}^{d_1}$.
Therefore there exists a bijection between
$N_{f_{x,x\vee d}^d}^{d_1,d_2}$ and $N_{f_{x',x'\vee d}^d}^{d_1,d_2}$.
Hence the triple $(\tilde P,P,\eel)$ is a schemoid. 
\end{proof}

Now we discuss the schemoid algebra.
Consider the polynomial ring $\KK[X_x|x\in P]$ in variables corresponding
to elements in $P$. 
Define $G_i$ by
\begin{align*}
 G_0&=\Set{X_0-1}\\
 G_1&=\Set{X_xX_y|\rho(x\vee y)<\rho(x)+\rho(y)}\\
 G'_1&=\Set{X_xX_y|\Set{ z|z\geq x, z\geq y}=\emptyset}\\
 G_2&=\Set{X_xX_y-X_{x\vee y}|\rho(x\vee y)=\rho(x)+\rho(y)}.
\end{align*} 
Let $I$ be the ideal generated by $G_0\cup G_1\cup G'_1\cup G_2$, and $R_P$
the quotient ring $\KK[X_x|x\in P]/I$.
The ring $R_P$ is the same as the ring defined in the following manner:
$R_P$ is the $\KK$-vector space whose basis is $\Set{X_x|x\in P}$
\begin{align*}
\begin{cases}
X_xX_y-X_{x\vee y} &(\text{if there exists $x\vee y$  and $\rho(x\vee y)=\rho(x)+\rho(y)$}),\\
 X_xX_y=0 &(\text{otherwise}).
\end{cases}
\end{align*}
\begin{theorem}
For the map $\eel$ from $\Mor(\tilde P)$ to $P$
defined by $\eel(f^d_{x,y})=d$, 
the schemoid algebra for the schemoid
$(\tilde P,P,\eel)$
is isomorphic to $R_P$. 
\end{theorem}

\begin{example}
Let $P$ be the $n$-th Boolean lattice $2^{[n]}$.
In this case,
the set of morphisms 
is 
\begin{align*}
\Set{f^{d}_{x,x\cup d}| x,d \subset [n],\ x\cap d =\emptyset}. 
\end{align*}
Hence $(\tilde P,P,\eel)$ is $(\CCC_P,P,\eel)$ in Example \ref{example:boolean}.
\end{example}

\begin{example}
Let $K$ be a finite field.
Consider the poset $P$ of all subspaces in $K^n$ 
ordered by the inclusion.
In this case,
the set of morphisms 
is 
\begin{align*}
\Set{f^{W}_{V,V\oplus W}| V,W\in P,\ V\cap W=0}. 
\end{align*}
\end{example}

\begin{example}
Let $P$ be the poset of flats of a matroid $M$ ordered by inclusion.
Assume that $P$ satisfies the conditions in this section.
In this case, 
the schemoid algebra for  $(\tilde P,P,\eel)$ is 
isomorphic to
the algebra defined in Maeno--Numata \cite{arXiv:1107.5094},
which is M\"obius algebra with
the relations $x^2_i=0$ for all variables $x_i$.
\end{example}

\end{document}